\documentclass[leqno]{amsart}
\usepackage{amssymb}
\usepackage{amsmath}
\usepackage{amsrefs}

\newtheorem{theorem}{Theorem}[section]
\newtheorem{lemma}[theorem]{Lemma}
\newtheorem{corollary}[theorem]{Corollary}

\newtheorem{definition}[theorem]{Definition}
\newtheorem{proposition}[theorem]{Proposition}

\numberwithin{equation}{section}

\DeclareMathOperator*{\osc}{osc}

\newcommand{ \mr }{ \mathbb{R} }
\newcommand{ \ba }{ \mathbf{a} }
\newcommand{ \m }{ \mathcal{M} }

\def\Xint#1{\mathchoice
    {\XXint\displaystyle\textstyle{#1}}%
     {\XXint\textstyle\scriptstyle{#1}}%
     {\XXint\scriptstyle\scriptscriptstyle{#1}}%
     {\XXint\scriptstyle\scriptscriptstyle{#1}}%
    \!\int}
\def\XXint#1#2#3{{\setbox0=\hbox{$#1{#2#3}{\int}$}
    \vcenter{\hbox{$#2#3$}}\kern-.5\wd0}}

\begin{document}

\title[Global Sobolev regularity]{Global Sobolev regularity for general elliptic equations of $p$-Laplacian type}

\author[S.-S. Byun, D.K. Palagachev, P. Shin]{Sun-Sig Byun\and Dian K. Palagachev\and Pilsoo Shin}

\address{Sun-Sig Byun: Seoul National University, Department of Mathematical Sciences and Research Institute of Mathematics, Seoul 151-747, Korea}
\email{byun@snu.ac.kr}

\address{Dian K. Palagachev: Politecnico di Bari,
Dipartimento di Meccanica, Matematica e Management,
Via Edoardo Orabona 4, 70125 Bari, Italy}
\email{dian.palagachev@poliba.it}

\address{Pilsoo Shin: Seoul National University, Department of Mathematical Sciences, Seoul 151-747, Korea}
\email{shinpilsoo@snu.ac.kr}



\begin{abstract}
We derive global gradient estimates for  $W^{1,p}_0(\Omega)$-weak solutions to quasilinear elliptic equations of the form
$$
\mathrm{div\,}\ba(x,u,Du)=\mathrm{div\,}(|F|^{p-2}F)
$$
over $n$-dimensional Reifenberg flat domains.
The nonlinear term of the elliptic differential operator is supposed to be small-BMO with respect to $x$ and H\"older continuous in $u.$ In the case when $p\geq n,$ we allow only continuous nonlinearity in $u.$

Our result highly improves the known regularity results available in the literature. In fact, we are able not only to weaken the regularity requirement on the nonlinearity in $u$ from Lipschitz continuity to H\"older one, but we also find a very lower level of geometric assumptions on the boundary of the domain to ensure global character of the obtained gradient estimates.
\end{abstract}

\maketitle

\section{Introduction}

Solutions to important real world problems from science and technology turn out to realize minimal energy of suitable nonlinear functionals. Finding these solutions and examining closely their qualitative properties is a central problem of the Calculus of Variations, and the machinery of the nonlinear functional analysis is what serves to pursue that study. On the other hand, each minimizer of a variational functional solves weakly the corresponding Euler--Lagrange equation and this fact allows to rely on the powerful theory of PDEs as additional tool in the Calculus of Variations. The Euler--Lagrange equations are divergence form PDEs of elliptic type, usually nonlinear, and their weak solutions (the minimizers) own some basic minimal smoothness. The \textit{regularity theory} of general (non necessary variational) divergence form elliptic PDEs establishes how the smoothness of the data of a given  problem influences the regularity of
the solution, obtained under very general circumstances. Once having better smoothness, powerful tools of
functional analysis apply to infer finer properties of the solution and the problem itself. The importance of these issues is even more evident in the settings of variational problems if dealing with discontinuous functionals over domains with
non-smooth boundaries when many of the classical nonlinear analysis techniques fail.

Starting with the deep results of Caffarelli and Peral (\cite{CP}), a notable progress has been achieved in the last two decades in the regularity theory of nonlinear divergence form elliptic PDEs (see also \cite{AM, BDM, BK, BW1, CM, KM, MP} and the references therein). On the base of suitable  $L^p$-estimates for the gradient $Du$ of the weak solution a satisfactory Calder\'on--Zygmund type theory has been developed, firstly for equations with principal term depending only on $Du,$ and later also dependence on the independent variables $x$ has been allowed. Moreover, the minimal regularity requirements have been identified for the nonlinear terms of the equations and the boundary of the underlying domain in order the Calder\'on--Zygmund theory still holds true for large class of equations with generally $x$-discontinuous ingredients. In all that context, the possibility to deal with equations with general nonlinearity with respect to the solution $u$ is a rather delicate matter,     and the reason of this lies in the fact that such equations are not invariant under particular scaling and normalization, whereas these are crucial ingredients of the perturbation approach in \cite{CP}.

\medskip

We deal here with the Dirichlet problem
\begin{equation}\label{1}
\begin{cases}
\mathrm{div\,}\ba(x,u,Du)=\mathrm{div\,}(|F|^{p-2}F) & \textrm{in}\ \Omega\\
u=0 & \textrm{on}\ \partial\Omega,
\end{cases}
\end{equation}
where $\Omega\subset\mr^n,$ $n\geq2,$ is a bounded and generally irregular domain,
$\ba\colon\mr^n\times\mr\times\mr^n\to\mr^n$ is a Carath\'eodory map,
$p>1$ is arbitrary exponent and $F\in L^p(\Omega,\mr^n).$

Our main goal is to obtain a Calder\'on--Zygmund type regularizing effect for \eqref{1}. Namely, assuming $F\in L^{p'}(\Omega,\mr^n)$ for $p'>p,$
 under rather general structure and regularity hypotheses on $\ba(x,z,\xi)$ and $\partial\Omega,$ we derive global $L^{p'}(\Omega)$-gradient estimate for any bounded $W^{1,p}_0(\Omega)$ weak solution of \eqref{1} in terms of $\|F\|_{L^{p'}(\Omega,\mr^n)}$ showing this way that $F\in L^{p'}(\Omega,\mr^n)$ implies $Du\in L^{p'}(\Omega,\mr^n).$

In the case when $\ba=\ba(x,Du),$ similar results have been obtained in \cite{BW2} in the settings of classical Lebesgue spaces and in \cite{BR} for weighted Lebesgue spaces, assuming the standard ellipticity condition and allowing discontinuity of $\ba$ with respect to $x,$ measured in terms of small-BMO seminorm. In the recent paper \cite{NP}, the authors succeeded to obtain \textit{interior} gradient estimates for  \eqref{1} also in the case when $\ba$ depends on the solution $u.$ The problems arising with the scaling and normalization in that situation are cleverly avoided by including the nonlinear differential operator into a \textit{two parameter} class of elliptic operators, that turns out to be invariant with respect to dilations and rescaling of domain. In order to run the approximation procedure of \cite{CP}, a \textit{uniform} control with respect to these two parameters is necessary, and the authors of \cite{NP} carry out it by means of a delicate compactness argument relying on the Minty trick. This approach, however, strongly requires \textit{uniqueness} for the approximating equation, that is why, $\ba(x,z,\xi)$ is assumed to be \textit{Lipschitz continuous with respect to $z$} in \cite{NP}.

Here we suppose that $\ba(x,z,\xi)$ is small-BMO function with respect to $x$ and it satisfies the standard uniform ellipticity condition in $\xi$ but, in contrast to \cite{NP}, $\ba$ is assumed to be \textit{only H\"older continuous} with respect to the variable $z.$ To get our main result, we combine the two-parameter approach from \cite{NP} with correct scaling arguments in the $L^q$-estimates for the maximal function of the gradient and Vitali type covering lemma. However, we rely here on the \textit{higher gradient integrability} in the spirit of Gehring--Giaquinta rather than on the uniqueness of the approximating equation, and this allows us to weaken the $z$-Lipschitz continuity of $\ba$ to only H\"older one.

We start with considering two appropriate reference problems with \textit{only gradient nonlinear terms}, given by the $z$-compositions of $\ba(x, z, \xi)$  first with the weak solution $u(x)$ and then with its local average $\bar{u}.$ Thanks to the uniform ellipticity of the associated
nonlinearities, the reference solutions support higher integrability results and H\"{o}lder continuity properties. We then combine these
properties with the $z$-H\"older continuity of $\ba$ and the comparison estimates of \cite{BR}, regarding nonlinear terms like $\ba(x,Du),$ in order to obtain the desired comparison estimates. Once having these, standard maximal
function approach and a Vitali type covering lemma give the main result.

It is worth noting that we need $\ba(x,z,\xi)$ to be
H\"older continuous in $z$ \textit{only in the case when $p<n.$} Otherwise, the weak solution of \eqref{1} is itself a H\"older continuous which implies that the nonlinear term in \eqref{1}, fixed at the solution $u(x),$ that is $\mathbf{A}(x,\xi):= \ba (x,u(x),\xi),$ is a small-BMO with respect to $x$ if $\ba (x,z,\xi)$ is required to be merely \textit{continuous} in $z.$ This suffices to run our procedure and to get the Calder\'on--Zygmund property assuming only $z$-\textit{continuity} of $\ba$ when $p\geq n.$

Another advantage of the approach here adopted is that it works also near the boundary of $\Omega$ and this allows to obtain \textit{global} gradient estimates for the solutions of \eqref{1}. Indeed, this requires some ``good'' geometric properties of $\partial\Omega$ and these are ensured when $\Omega$ belongs to the class of the Reifenberg flat domains.

The paper is organized as follows. In Section~\ref{Sec2} we list the hypotheses imposed on the data and state the main result, Theorem~\ref{Thm1}. Some comments about the structure and regularity assumptions required are given as well. Section~\ref{Sec3} provides an analysis of how the equation in \eqref{1} and the hypotheses on the nonlinear term behave under the two-parameter scaling and normalization.
Section~\ref{Sec4} forms the analytic heart of the paper. We derive there good gradient estimates for solutions to appropriate limiting problems to which \eqref{1} compares. With these estimates at hand, we
employ in Section~\ref{Sec5} a Vitali type covering lemma and scaling arguments in order to prove Theorem~\ref{Thm1} by obtaining suitable decay estimates for the level sets of the Hardy--Littlewood maximal function of the gradient. The last Section~\ref{Sec6} is devoted to the refinement of the main result in the case when $p\geq n.$ The H\"older continuity of $\ba(x,z,\xi)$ with respect to $z$ is relaxed to only continuity and we combine our recent results \cite{BPS,BPS-arxiv} with these of \cite{BR} to get the refined  version of Theorem~\ref{Thm1} when $p\geq n.$

\bigskip

\noindent
{\bf Acknowledgements.}

S.-S.~Byun was supported by the National Research Foundation of Korea  grant funded by the Korea government (MEST)  (NRF-2015R1A4A1041675).
 D.K.~Palagachev is member of the Gruppo Nazionale per l'Analisi Matematica, la Probabilit\`a e le loro Applicazioni (GNAMPA) of the Istituto Nazionale di Alta Matematica (INdAM).
P. Shin was supported by National Research Foundation of Korea grant funded by the Korean government (MEST) (NRF-2015R1A2A1A15053024).

\section{Hypotheses and main results}\label{Sec2}

Throughout the paper, we will use standard notations and will assume that the functions and sets considered are measurable.

We denote by $B_\rho(x)$ (or simply $B_\rho$ if there is no ambiguity) the $n$-dimensional open ball with center $x\in\mr^n$ and radius $\rho,$
and $\Omega_\rho(x):=\Omega\cap B_\rho(x)$
 for an open set $\Omega\subset \mr^n.$
 The Lebesgue measure of a measurable set $A\subset\mr^n$ will be denoted
by $|A|$ while,  for any integrable function $u$ defined on $A,$
$$
\overline{u}_A:=\Xint-_A u(x)\; dx = \frac{1}{|A|}\int_A u(x)\; dx
$$
stands for its integral average. If
 $u\in L^1_{\text{loc}}(\mr^n),$ then the Hardy--Littlewood maximal function of $u$ is given by
$$
\m u(x) : = \sup_{\rho>0} \Xint-_{B_\rho(x)} |u(y)|\; dy,
$$
while $\m_A u:=\m\left(\chi_A u\right)$ when
  $u$ is defined on a measurable set $A,$ with the characteristic function  $\chi_A$ of the set $A.$

We will denote by $C^\infty_0(\Omega)$ the space of infinitely differentiable functions over a bounded domain $\Omega\subset\mr^n$ with compact support contained in $\Omega,$ and $L^p(\Omega)$ stands for the standard Lebesgue space with a given $p\in[1,\infty].$ The Sobolev space $W^{1,p}_0(\Omega)$ is defined, as usual, by the completion of $C^\infty_0(\Omega)$ with respect to the norm
$$
\|u\|_{W^{1,p}(\Omega)}:=\|u\|_{L^p(\Omega)}+\|Du\|_{L^p(\Omega)}
$$
for $p\in [1,\infty).$

In what follows we will consider a bounded domain $\Omega\subset\mr^n$ with $n\geq2,$ the boundary $\partial\Omega$ of which is \textit{Reifenberg flat} in the sense of the following definition.

\begin{definition}\label{defin2.1}
The domain $\Omega$ is said to be $(\delta,R)$-Reifenberg flat if there exist positive constants $\delta$ and $R$ with the property that for each $x_0\in\partial\Omega$ and each $\rho\in(0,R)$ there is a local coordinate system $\{x_1,\cdots,x_n\}$ with origin at the point $x_0,$ and such that
$$
B_\rho(x_0)\cap\{x:x_n>\rho\delta\}\subset B_\rho(x_0)\cap\Omega \subset B_\rho(x_0)\cap\{x:x_n>-\rho\delta\}.
$$
\end{definition}

Turning back to  problem \eqref{1}, the nonlinear term is given by the Carath\'eodory map $\ba\colon\Omega\times\mr\times\mr^n\to\mr^n$ where
$\ba(x,z,\xi)=\big(a^1(x,z,\xi),\cdots,a^n(x,z,\xi)\big).$ We suppose moreover that
$\ba(x,z,\xi)$ is differentiable with respect to $\xi,$ and $D_\xi\ba$ is a Carath\'eodory map.

Throughout the paper the following structure and regularity conditions on the data will be assumed:

\medskip

$\bullet$ \textit{Uniform ellipticity:} There exists a constant $\gamma>0$ such that
\begin{equation}\label{3}
\begin{cases}
\gamma|\xi|^{p-2}|\eta|^2\leq\langle D_\xi\ba(x,z,\xi)\eta,\eta\rangle,\\
|\ba(x,z,\xi)|+|\xi||D_\xi\ba(x,z,\xi)|\leq\gamma^{-1}|\xi|^{p-1}
\end{cases}
\end{equation}
for a.a. $x\in\Omega$ and $\forall(z,\xi)\in\mr\times\mr^n,$ $\forall\eta\in\mr^n.$

It is worth noting that the uniform ellipticity condition \eqref{3} implies easily  the following monotonicity property:
\begin{equation}\label{Mo1}
\big\langle \ba(x,z,\xi_1)-\ba(x,z,\xi_2),\xi_1-\xi_2 \big\rangle \geq
\begin{cases}
\widetilde{\gamma}|\xi_1-\xi_2|^p & \textrm{if}\ p\geq 2,\\
\widetilde{\gamma}|\xi_1-\xi_2|^2(|\xi_1|+|\xi_2|)^{p-2} & \textrm{if}\ 1<p<2,
\end{cases}
\end{equation}
where $\widetilde{\gamma}$ depends only on $\gamma,$ $n$ and $p.$

\medskip

$\bullet$ \textit{H$\ddot{o}$lder continuity:} There exist constants $\Gamma>0$ and $0<\alpha<1$ such that
\begin{equation}\label{4}
|\ba(x,z_1,\xi)-\ba(x,z_2,\xi)|\leq \Gamma|z_1-z_2|^\alpha|\xi|^{p-1}
\end{equation}
for a.a. $x\in\Omega,$ $\forall z_1,z_2\in\mr$ and $\forall\xi\in\mr^n.$

\medskip

$\bullet$ \textit{$(\delta,R)$-vanishing property:} For each constant $M>0$ there exist $R>0$ and $\delta>0,$ depending on $M,$ such that
\begin{equation}\label{5}
\sup_{z\in[-M,M]}
\sup_{0<\rho\leq R}\sup_{y\in\mr^n}\Xint-_{B_\rho(y)}\Theta\big(\ba;B_\rho(y)\big)(x,z)dx \leq \delta,
\end{equation}
where the function $\Theta$ is defined by
$$
\Theta\big(\ba;B_\rho(y)\big)(x,z):=\sup_{\xi\in\mr^n\setminus\{0\}}\frac{|\ba(x,z,\xi)-\overline{\ba}_{B_\rho(y)}(z,\xi)|}{|\xi|^{p-1}},
$$
and $\overline{\ba}_{B_\rho(y)}(z,\xi)$ is the integral average of $\ba(x,z,\xi)$ in the variables $x$ for a fixed couple $(z,\xi)\in\mr\times\mr^n,$ that is,
$$\overline{\ba}_{B_\rho(y)}(z,\xi)=\Xint-_{B_\rho(y)}\ba(x,z,\xi)\ dx.
$$

\medskip

To make clear the meaning of the above assumptions, we should note that, thanks to the scaling invariance property of $\partial\Omega,$   $R$ could be any number greater than $1$ in Definition~\ref{defin2.1}, while $R$ could be taken equal to $\text{diam\,}\Omega$ in \eqref{5}. For what concerns
 $\delta$ instead, the definitions of $(\delta,R)$-Reifenberg flatness and $(\delta,R)$-vanishing property are significant only for small values, say $\delta\in(0,1/8).$ Roughly speaking, the Reifenberg flatness of $\Omega$ means that $\partial\Omega$ is well approximated by hyperplanes at every point and at every scale. In particular, domains  with $C^1$-smooth boundary or with boundary that is locally given as graph
of a Lipschitz continuous function with small Lipschitz constant are Reifenberg flat.
Actually, the class of the Reifenberg flat domains is much wider and contains sets with rough fractal boundaries such as the von Koch snowflake that is a Reifenberg flat when the angle of the spike with respect to the horizontal is small enough. As for the $(\delta,R)$-vanishing property \eqref{5}, it exhibits a sort of smallness in terms of BMO for what concerns the behaviour of $\ba(x,z,\xi)$ with respect to the $x$-variables. For instance, \eqref{5} is satisfied when $\ba\in C^0_x$ or even VMO${}_x.$ This way, \eqref{5} allows $x$-\textit{discontinuity} of the nonlinearity which is controlled in terms of small-BMO.

\medskip

Turning bach to the Dirichlet problem \eqref{1}, recall that a function $u\in W^{1,p}_0(\Omega)$ is said to be a  \textit{weak solution}  if
$$
\int_\Omega \big\langle \ba(x,u(x),Du(x)), D\phi(x)\big\rangle\; dx = \int_\Omega \big\langle|F(x)|^{p-2}F(x), D\phi(x)\big\rangle\; dx
$$
for each test function $\phi\in W^{1,p}_0(\Omega).$

\medskip

Our main result is as follows.
\begin{theorem}\label{Thm1}
Suppose \eqref{3} and \eqref{4}, and let $u\in W^{1,p}_0(\Omega)\cap L^\infty(\Omega)$ be a bounded weak solution of \eqref{1}. Assume that  $\|u\|_{L^\infty(\Omega)}\leq M$ and $|F|^p\in L^q(\Omega)$ for some $q\in (1,\infty).$ Then there is a small constant $\delta=\delta(\gamma,\alpha,n,p,q,\Gamma,M)$ such that if $\ba$ is $(\delta,R)$-vanishing and $\Omega$ is $(\delta,R)$-Reifenberg flat, then $|Du|^p\in L^q(\Omega)$ with the estimate
$$
\int_\Omega |Du|^{pq} dx \leq C \int_\Omega |F|^{pq} dx,
$$
where $C>0$ depends only on $\gamma,$ $\alpha,$ $n,$ $p,$ $q,$ $\Gamma,$ $M$ and $|\Omega|.$
\end{theorem}

\section{Scaling and normalization properties}\label{Sec3}

In this section, we will show how the scaling and normalization reflect on the structure conditions and regularity assumptions imposed on the data.

Recall that that $\Omega$ is assumed to be a $(\delta,R)$-Reifenberg flat domain and the nonlinearity $\ba$ satisfies the conditions  \eqref{3}, \eqref{4} and the $(\delta,R)$-vanishing property \eqref{5}. Let  $\sigma$ be a large enough positive constant which is to be determined later in a universal way so that it will depend only on the given data such as $n,$ $p,$ $q,$ $\gamma,$ $\Gamma$ and $M.$  Then for each fixed $\lambda>0$ and $0<r\leq \frac{R}{\sigma},$ we define a bounded domain
$$
\widetilde{\Omega}=\left\{\frac{1}{r}x:x\in\Omega\right\},
$$
a Carath\'eodory map $\widetilde{\ba}\colon\mr^n\times\mr\times\mr^n\to\mr^n$ by
$$
\widetilde{\ba}(x,z,\xi)=\frac{\ba(r x,\lambda r z,\lambda\xi)}{\lambda^{p-1}},
$$
a Sobolev function $\widetilde{u}\in W^{1,p}_0(\widetilde{\Omega})$ and a measurable function $\widetilde{F}\in L^{p}(\widetilde{\Omega})$ by
$$
\widetilde{u}(x)=\frac{u(rx)}{\lambda r}\quad \mathrm{and}\quad \widetilde{F}(x)=\frac{F(r x)}{\lambda}.
$$

Straightforward calculations yield the following properties:
\begin{itemize}
\item[$\bullet$]
$\widetilde{\ba}$ satisfies the uniform ellipticity condition \eqref{4} with the same constant $\gamma.$ That is,
\begin{equation}\label{S}
\begin{cases}
\gamma|\xi|^{p-2}|\eta|^2\leq\langle D_\xi\widetilde{\ba}(x,z,\xi)\eta,\eta\rangle,\\
|\widetilde{\ba}(x,z,\xi)|+|\xi||D_\xi\widetilde{\ba}(x,z,\xi)|\leq\gamma^{-1}|\xi|^{p-1}
\end{cases}
\end{equation}
for a.a. $x\in\mr^n$ and $\forall(z,\xi)\in\mr\times\mr^n,$ $\forall\eta\in\mr^n.$

Moreover, the monotonicity
\begin{align}\label{Mo}
\langle \widetilde{\ba}(x,z,\xi_1)-&\widetilde{\ba}(x,z,\xi_2),\xi_1-\xi_2 \rangle\\
\nonumber
 &\geq
\begin{cases}
\widetilde{\gamma}|\xi_1-\xi_2|^p & \textrm{if}\ p\geq 2,\\
\widetilde{\gamma}|\xi_1-\xi_2|^2(|\xi_1|+|\xi_2|)^{p-2} & \textrm{if}\ 1<p<2,
\end{cases}
\end{align}
does follow with the same constant $\widetilde{\gamma}.$

\item[$\bullet$]
$\widetilde{\ba}$ satisfies
\begin{equation}\label{S1}
|\widetilde{\ba}(x,z_1,\xi)-\widetilde{\ba}(x,z_2,\xi)|\leq \Gamma(\lambda r)^\alpha|z_1-z_2|^\alpha|\xi|^{p-1}
\end{equation}
for a.a. $x\in\widetilde{\Omega},$ $\forall z_1,z_2\in\mr$ and $\forall\xi\in\mr^n$ with the same constants $\alpha$ and $\Gamma.$

\item[$\bullet$]
$\widetilde{\ba}$ is $(\delta,\frac{R}{r})$-vanishing. Namely,
$$
\sup_{z\in[-\frac{M}{\lambda r},\frac{M}{\lambda r}]}
\sup_{0<\rho\leq \frac{R}{r}}\sup_{y\in\mr^n}\Xint-_{B_\rho(y)}\Theta\big(\widetilde{\ba};B_\rho(y)\big)(x,z)dx \leq \delta.
$$

\item[$\bullet$]
$\widetilde{\Omega}$ is $(\delta,\frac{R}{r})$-Reifenberg flat.

\item[$\bullet$]
If $u\in W^{1,p}_0(\Omega)$ is a weak solution of \eqref{1}, then $\widetilde{u}\in W^{1,p}_0(\widetilde{\Omega})$ is a weak solution of the problem
\begin{equation}\label{UP}
\begin{cases}
\mathrm{div\,}\widetilde{\ba}(x,\widetilde{u}(x),D\widetilde{u}(x))=\mathrm{div\,}(|\widetilde{F}|^{p-2}\widetilde{F}) & \textrm{in}\ \widetilde{\Omega}\\
\widetilde{u}=0 & \textrm{on}\ \partial\widetilde{\Omega}.
\end{cases}
\end{equation}
\end{itemize}

\section{Comparison estimates}\label{Sec4}

A crucial step in the proof of the main result is ensured by appropriate comparison of the  weak solution to \eqref{UP} with  these of  the associated reference problems  \eqref{HP}, \eqref{FP} and \eqref{VP} below. Throughout the section, for the sake of simplicity,
we will use the notations $u,$ $F,$ $\ba$ and $\Omega,$ instead of  $\widetilde{u},$ $\widetilde{F},$ $\widetilde{\ba}$ and $\widetilde{\Omega},$ respectively.

We start with the following useful lemma.
\begin{lemma}\label{Lem1}
Let $\Omega\subset\mr^n$ be a bounded open set. Assume that $\ba:\Omega\times\mr\times\mr^n\to\mr^n$ satisfies \eqref{Mo} for a.a. $x\in\Omega$ and for some $p\in(1,2).$ Then, for any $\zeta_1,$ $\zeta_2\in W^{1,p}(\Omega_\rho),$ any non-negative function $\eta\in C^\infty(B_\rho),$ any bounded function $\phi$ defined on $\Omega_\rho$ and any constant $\tau>0,$ we have
\begin{align*}
\int_{\Omega_\rho} |D\zeta_1-D\zeta_2|^p\eta\; dx \leq &\tau\int_{\Omega_\rho} |D\zeta_1|^p\eta\; dx\\
&+ C\int_{\Omega_\rho} \langle \ba(x,\phi,D\zeta_1)-\ba(x,\phi,D\zeta_2),D\zeta_1-D\zeta_2\rangle\eta\; dx
\end{align*}
with $C>0$ depending only on $\gamma,$ $p$ and $\tau.$
\end{lemma}

\begin{proof}
See \cite[the proof of Lemma~3.7]{BR}, \cite[Lemma~3.1]{NP}.
\end{proof}

Let $\sigma > 6$ be a universal constant which will be chosen later in Lemma \ref{Lem3}, and consider a localized solution $u$ in $\Omega_\sigma$ of the problem
\begin{equation}\label{LUP}
\begin{cases}
\mathrm{div\,}\ba(x,u(x),Du(x))=\mathrm{div\,}(|\widetilde{F}|^{p-2}\widetilde{F}) & \textrm{in}\ \Omega_\sigma,\\
u=0 & \textrm{on}\ \partial \Omega\cap B_\sigma,
\end{cases}
\end{equation}
with
\begin{align}\label{31}
\|u\|_{L^\infty(\Omega_\sigma)}\leq \frac{M}{\lambda r}, \quad \frac{1}{|B_\sigma|}\int_{\Omega_\sigma} |Du|^p dx \leq 1,\quad \mathrm{and}\quad \frac{1}{|B_\sigma|}\int_{\Omega_\sigma} |F|^p dx \leq \delta^p.
\end{align}
Assume further that
\begin{align}\label{32}
\frac{1}{|B_6|}\int_{\Omega_6} |Du|^p dx \leq 1.
\end{align}

We let next  $h\in W^{1,p}(\Omega_\sigma)$ to be the weak solution of
\begin{equation}\label{HP}
\begin{cases}
\mathrm{div\,}\ba(x,u(x),Dh(x))=0 & \textrm{in}\ \Omega_\sigma,\\
h=u & \textrm{on}\ \partial \Omega_\sigma,
\end{cases}
\end{equation}
and $f\in W^{1,p}(\Omega_5)$  the weak solution of
\begin{equation}\label{FP}
\begin{cases}
\mathrm{div\,}\ba(x,\overline{u}_{\Omega_5},Df(x))=0 & \textrm{in}\ \Omega_5,\\
f=h & \textrm{on}\ \partial \Omega_5,
\end{cases}
\end{equation}
with
\begin{equation}\label{33}
\frac{1}{|B_5|}\int_{\Omega_5}\Theta\big(\ba;\Omega_5\big)(x,\overline{u}_{\Omega_5})dx \leq \delta
\end{equation}
and
\begin{equation}\label{34}
B_5^+\subset \Omega_5 \subset B_5\cap\{x:x_n>-10\delta\}.
\end{equation}

We consider finally the  limiting problem
\begin{equation}\label{VP}
\mathrm{div\,}\mathbf{A}(Dv(x))=0 \quad \textrm{in}\ U,
\end{equation}
where $U=B_4$ for the interior case  and $U=B_4^+$ for the boundary case, where the map $\mathbf{A}:\mr^n \to \mr^n$ is given by
$$
\mathbf{A}(\xi):=\frac{1}{|U|}\int_U \ba(x,\overline{u}_{\Omega_5},\xi)\; dx.
$$

\medskip

The following is the main result of this section.

\begin{lemma}\label{MainLem}
For any small constant $\varepsilon\in(0,1),$ there exist  a large constant $\sigma=\sigma(\gamma,\alpha, n,p,\Gamma,M,\varepsilon)> 6$ and a small constant $\delta=\delta(\gamma,\alpha, n,p,\Gamma,M,\varepsilon)>0$ such that if $u\in W^{1,p}(\Omega_\sigma)$ is a weak solution of \eqref{LUP} with
 \eqref{31}, \eqref{32}, \eqref{33} and \eqref{34},
then there exists a weak solution $v\in W^{1,p}(U)$ of \eqref{VP} such that
$$
\|D\bar{v}\|_{L^\infty(\Omega_3)}\leq N_0\quad \text{and}\quad\int_{\Omega_4} |Du-D\bar{v}|^p dx \leq \varepsilon^p
$$
for some constant $N_0=N_0(\gamma,n,p)>1.$ Here, the function $\bar{v}\in W^{1,p}(\Omega_4)$ is equal to
$v$ when $U=B_4,$ and $\bar{v}$ is the zero extension of $v$ from $B_4^+$ to $B_4$ when $U=B_4^+.$
\end{lemma}

The proof  is based on the following Lemmas~\ref{Lem2}, \ref{Lem3} and \ref{Lem4}.

\begin{lemma}\label{Lem2}
Let $\sigma>6$ and $0<\varepsilon_1<1.$  Then there exists a small constant $\delta=\delta(n,p,\gamma,\sigma,\varepsilon_1)>0$ such that if $u\in W^{1,p}(\Omega_\sigma)$ is a weak solution of \eqref{LUP} and
 $h \in W^{1,p}(\Omega_\sigma)$ is the weak solution of \eqref{HP} with \eqref{31},
then
\begin{align}\label{42}
\frac{1}{|B_6|}\int_{\Omega_6} |Du-Dh|^p dx \leq \varepsilon_1^p.
\end{align}
\end{lemma}

\begin{proof}
The proof will be divided into two cases.

\medskip

$\mathbf{Case\ 1:}\ 1<p<2.$
Taking $u-h$ as a test function for equations \eqref{LUP} and \eqref{HP},  it follows from the Young inequality with $\tau_1>0$ that
\begin{align}\label{43}
\int_{\Omega_\sigma} \langle \ba(x,u,Du)-\ba(x,u,Dh),Du-Dh\rangle\; dx = \int_{\Omega_\sigma} \langle |F|^{p-2}F, Du-Dh \rangle\; dx\\
\nonumber \leq \tau_1\int_{\Omega_\sigma}|Du-Dh|^p\; dx + C(\tau_1)\int_{\Omega_\sigma} |F|^p\; dx.
\end{align}
Then Lemma~\ref{Lem1} implies
\begin{align*}
\int_{\Omega_\sigma} &|Du-Dh|^p\; dx\\
&\leq \tau\int_{\Omega_\sigma} |Du|^p\; dx + C_0(\tau)\int_{\Omega_\sigma} \langle \ba(x,u,Du)-\ba(x,u,Dh),Du-Dh\rangle\; dx\\
&\leq \tau\int_{\Omega_\sigma} |Du|^p\; dx + C_0\tau_1\int_{\Omega_\sigma}|Du-Dh|^p\; dx + C(\tau,\tau_1)\int_{\Omega_\sigma} |F|^p\; dx.
\end{align*}
Setting $\tau_1=\frac{1}{2C_0}$ in the above inequality, we obtain
\begin{align*}
\int_{\Omega_\sigma} &|Du-Dh|^p\; dx \leq 2\tau\int_{\Omega_\sigma} |Du|^p\; dx +  C(\tau)\int_{\Omega_\sigma} |F|^p\; dx,
\end{align*}
and so \eqref{31} yields
\begin{align*}
\frac{1}{|B_6|}\int_{\Omega_6} |Du-Dh|^p\; dx &\leq \left(\frac{\sigma}{6}\right)^n\frac{1}{|B_\sigma|}\int_{\Omega_\sigma} |Du-Dh|^p\; dx\\
&\leq \left(\frac{\sigma}{6}\right)^n\frac{2\tau}{|B_\sigma|}\int_{\Omega_\sigma} |Du|^p\; dx + C(\tau,\sigma)\frac{1}{|B_\sigma|}\int_{\Omega_\sigma} |F|^p\; dx\\
&\leq 2\tau\left(\frac{\sigma}{6}\right)^n + C(\tau,\sigma)\delta^p.
\end{align*}
Now, taking the constants $\tau$ and $\delta$ sufficiently small so that
$$2\tau\left(\frac{\sigma}{6}\right)^n \leq \frac{\varepsilon_1^p}{2}\quad \mathrm{and} \quad C(\tau,\sigma)\delta^p \leq \frac{\varepsilon_1^p}{2},$$
we obtain the conclusion \eqref{42} when $1<p<2.$

\medskip

$\mathbf{Case\ 2:}\ p\geq 2.$
Having in mind  \eqref{Mo} and \eqref{43}, we get
\begin{align*}
\int_{\Omega_\sigma} |Du-Dh|^p\; dx &\leq \widetilde{\gamma}^{-1}\int_{\Omega_\sigma} \langle \ba(x,u,Du)-\ba(x,u,Dh),Du-Dh\rangle\; dx\\
&\leq  \widetilde{\gamma}^{-1}\tau_1\int_{\Omega_\sigma}|Du-Dh|^p\; dx + C(\tau_1)\int_{\Omega_\sigma} |F|^p\; dx
\end{align*}
for any $\tau_1>0.$ Taking $\tau_1=\frac{\widetilde{\gamma}}{2}$ in the above inequality,  it follows from \eqref{31} that
\begin{align*}
\frac{1}{|B_6|}\int_{\Omega_6} |Du-Dh|^p\; dx &\leq \left(\frac{\sigma}{6}\right)^n\frac{1}{|B_\sigma|}\int_{\Omega_\sigma} |Du-Dh|^p\; dx\\
&\leq \frac{C_1(\gamma,p)\sigma^n}{|B_\sigma|}\int_{\Omega_\sigma} |F|^p\; dx \leq C_1\sigma^n\delta^p.
\end{align*}

We choose now the constant $\delta$  small enough to have $C_1\sigma^n\delta^p \leq \varepsilon_1^p,$ and this gives the claim
of Lemma~\ref{Lem2}.
\end{proof}

We need the following higher integrability result for the equation \eqref{HP}.

\begin{proposition}\label{LemHIG}
(\cite[Theorem~1.1]{KK}, \cite[Theorem~2.2]{BW2} \cite[Lemma~3.2]{BR})
Let $h\in W^{1,p}(\Omega_\sigma)$ be a solution of \eqref{HP}. Then there is a positive constants $p_0>p$ depending only on $\gamma,$ $n$ and $p$ such that for any $p_1\in (p,p_0),$
$$
\left(\int_{\Omega_5} |Dh|^ {p_1}\ dx\right)^\frac{1}{p_1} \leq C\left(\int_{\Omega_6} |Dh|^p\ dx\right)^\frac{1}{p}
$$
holds, where $C>0$ depends only on $\gamma,$ $n,$ $p$ and $p_0.$
\end{proposition}

We also need the following oscillation theorem for the equation \eqref{HP}.

\begin{proposition}
\label{LemHOL}
(\cite[Theorem~4.2]{T}, \cite[Theorem~7.7]{G})
Let $h\in W^{1,p}(\Omega_\sigma)$ be a solution of \eqref{HP}. Then there is a positive constant $\beta>0$ depending only on  $\gamma,$ $n$ and $p$ such that
$$
\osc_{\Omega_5} h \leq C\left(\frac{5}{\sigma}\right)^\beta \|h\|_{L^\infty(\Omega_\sigma)}
$$
holds, where $C>0$ depend only on $\gamma,$ $n$ and $p.$
\end{proposition}

Now, we compare the weak solution $h\in W^{1,p}(\Omega_\sigma)$ of \eqref{HP} with the weak solution $f\in W^{1,p}(\Omega_5)$ of \eqref{FP} to have the following result.

\begin{lemma}\label{Lem3}

For any $\varepsilon>0,$ there are two constants $\delta\in\left(0,\frac{1}{8}\right)$ and $\sigma>6$ depending only on $\gamma,$ $n,$ $p,$ $\Gamma,$ $\alpha,$ $M$ and $\varepsilon,$ such that if $u\in W^{1,p}(\Omega_\sigma)$ is a weak solution of \eqref{LUP} and $f\in W^{1,p}(\Omega_5)$ is the weak solution of \eqref{VP} with \eqref{31}, \eqref{32} and \eqref{34}, then
$$
\int_{\Omega_5} |Dh-Df|^p dx \leq \varepsilon^p.
$$
\end{lemma}

\begin{proof}
The proof will be divided into two cases.

\medskip

$\mathbf{Case\ 1:}\ 1<p<2.$
We first prove the following inequality:
\begin{align}\label{45}
\int_{\Omega_5} |Dh-Df|^p dx \leq C_0\int_{\Omega_6} |h|^p dx
\end{align}
where $C_0$ depends only on $\gamma,$ $n$ and $p.$

Let $\eta\in C^\infty_0(B_6)$ be a cut-off function with the properties $0\leq\eta\leq1,$ $\eta\equiv1$ on $B_5$ and $|D\eta|\leq 2.$ Taking $\eta^ph$ as a test function for the equation \eqref{HP}, we have
\begin{align*}
\int_{\Omega_6} \langle \ba(x,u,Dh)-\ba(x,u,0),\eta^pDh\rangle\; dx &= \int_{\Omega_6} \langle \ba(x,u,Dh),\eta^pDh\rangle\; dx\\
&= -p\int_{\Omega_6} \langle \ba(x,u,Dh),\eta^{p-1}hD\eta\rangle\; dx\\
&\leq \gamma^{-1}p\int_{\Omega_6} \eta^{p-1}|h||Dh|^{p-1}|D\eta|\; dx\\
&\leq \tau\int_{\Omega_6} \eta^p|Dh|^p\; dx + C(\tau) \int_{\Omega_6} |h|^p|D\eta|^p\; dx
\end{align*}
as consequence of the Young inequality with $\tau>0.$ By Lemma~\ref{Lem1}, we have
\begin{align*}
\int_{\Omega_6} \eta^p|Dh|^p\; dx &\leq \frac{1}{4}\int_{\Omega_6} \eta^p|Dh|^p\; dx + C\int_{\Omega_6} \langle \ba(x,u,Dh)-\ba(x,u,0),\eta^pDh\rangle\; dx\\
&\leq \frac{1}{2}\int_{\Omega_6} \eta^p|Dh|^p\; dx + C\int_{\Omega_6} |h|^p|D\eta|^p\; dx,
\end{align*}
and so
\begin{align}\label{46}
\int_{\Omega_5} |Dh|^p\; dx \leq C\int_{\Omega_6} |h|^p\; dx.
\end{align}

Further on, taking $h-f$ as a test function for \eqref{HP} and \eqref{FP}, we obtain
 \begin{align}\label{47}
\int_{\Omega_6} \langle \ba(x,u,Dh),Dh-Df \rangle\; dx = \int_{\Omega_6} \langle \ba(x,\overline{u}_{\Omega_5},Df),Dh-Df \rangle\; dx.
\end{align}
In view of Lemma~\ref{Lem1} with $\eta\equiv 1,$ \eqref{S} and the Young inequality, we obtain that
\begin{align*}
\int_{\Omega_5} |Dh-&Df|^p dx\\
&\leq C\int_{\Omega_5} |Dh|^p dx + C\int_{\Omega_5} \langle \ba(x,\overline{u}_{\Omega_5},Dh)-\ba(x,\overline{u}_{\Omega_5},Df),Dh-Df\rangle dx\\
&= C\int_{\Omega_5} |Dh|^p dx + C\int_{\Omega_5} \langle \ba(x,\overline{u}_{\Omega_5},Dh)-\ba(x,u,Dh),Dh-Df\rangle dx \\
&\leq C\int_{\Omega_5} |Dh|^p dx + C\int_{\Omega_5}|Dh|^{p-1}|Dh-Df| dx\\
&\leq \frac{1}{2}\int_{\Omega_5} |Dh-Df|^p dx + C\int_{\Omega_5} |Dh|^p dx.
\end{align*}
Thus, the claim \eqref{45} follows
by \eqref{46}.

Recalling $\|u\|_{L^\infty(\Omega_6)}\leq \frac{M}{\lambda r},$ the maximum principle implies $\|h\|_{L^\infty(\Omega_6)}\leq \frac{M}{\lambda r}.$ Therefore, \eqref{45} yields
$$
\int_{\Omega_5} |Dh-Df|^p\; dx \leq C_0\int_{\Omega_6} |h|^p\; dx \leq C_0|B_6|\left(\frac{M}{\lambda r}\right)^p.
$$

If $C_0|B_6|\left(\frac{M}{\lambda r}\right)^p \leq \varepsilon^p,$ then we get the conclusion.

So, assume alternatively that $C_0|B_6|\left(\frac{M}{\lambda r}\right)^p > \varepsilon^p.$
In view of \eqref{S1}, \eqref{47} and Lemma~\ref{Lem1}, we have
\begin{align*}
\int_{\Omega_5} |Dh-&Df|^p dx\\
&\leq \frac{\tau}{2}\int_{\Omega_5} |Dh|^p dx + C(\tau)\int_{\Omega_5} \langle \ba(x,\overline{u}_{\Omega_5},Dh)-\ba(x,\overline{u}_{\Omega_5},Df),Dh-Df\rangle dx\\
&= \frac{\tau}{2}\int_{\Omega_5} |Dh|^p dx + C(\tau)\int_{\Omega_5} \langle \ba(x,\overline{u}_{\Omega_5},Dh)-\ba(x,u,Dh),Dh-Df\rangle dx \\
&\leq \frac{\tau}{2}\int_{\Omega_5} |Dh|^p dx + C(\tau)\int_{\Omega_5} (\lambda r)^\alpha|u-\overline{u}_{\Omega_5}|^\alpha|Dh|^{p-1}|Dh-Df| dx.
\end{align*}
The Young inequality gives
\begin{align*}
\int_{\Omega_5} &|Dh-Df|^p dx\\
&\leq \frac{\tau}{2}\int_{\Omega_5} |Dh|^p dx + C(\tau)\int_{\Omega_5} (\lambda r)^\frac{\alpha p}{p-1}|u-\overline{u}_{\Omega_5}|^\frac{\alpha p}{p-1}|Dh|^{p} dx + \frac{1}{2} \int_{\Omega_5} |Dh-Df|^p dx,
\end{align*}
and this implies
\begin{align}\label{48}
\int_{\Omega_5} |Dh-Df|^p dx  \leq \tau\int_{\Omega_5} |Dh|^p dx + C(\tau)\int_{\Omega_5} (\lambda r|u-\overline{u}_{\Omega_5}|)^\frac{\alpha p}{p-1}|Dh|^{p} dx.
\end{align}

To estimate the second term in the above inequality, we first take constants $\alpha_0$ and $p_1$ such that
$$
0<\alpha_0<\min\{\alpha,\ p-1\},\quad p<p_1<p_0, \quad \mathrm{and}\quad p_1=\frac{p(p-1)}{p-\alpha_0-1},
$$
where $\alpha$ is given in \eqref{S1} and $p_0$ is as in Proposition~\ref{LemHIG}.
We then use the H\"older inequality, \eqref{31} and Proposition~\ref{LemHIG}, to find that
\begin{align*}
\int_{\Omega_5} &(\lambda r|u-\overline{u}_{\Omega_5}|)^\frac{\alpha p}{p-1}|Dh|^{p} dx\\
 &\leq \left(\int_{\Omega_5} (\lambda r|u-\overline{u}_{\Omega_5}|)^\frac{\alpha pp_1}{(p-1)(p_1-p)} dx\right)^\frac{p_1-p}{p_1}\left(\int_{\Omega_5}|Dh|^{p_1} dx\right)^\frac{p}{p_1}\\
&\leq (2M)^\frac{(\alpha-\alpha_0)p}{p-1}\left(\int_{\Omega_5} (\lambda r|u-\overline{u}_{\Omega_5}|)^\frac{\alpha_0 pp_1}{(p-1)(p_1-p)} dx\right)^\frac{p_1-p}{p_1}\left(\int_{\Omega_5}|Dh|^{p_1} dx\right)^\frac{p}{p_1}\\
&\leq C \left(\int_{B_5} (\lambda r|u-\overline{u}_{\Omega_5}|)^\frac{\alpha_0 pp_1}{(p-1)(p_1-p)} dx\right)^\frac{p_1-p}{p_1}\int_{\Omega_6}|Dh|^{p} dx\\
&= C \left(\int_{\Omega_5} (\lambda r|u-\overline{u}_{\Omega_5}|)^p dx\right)^\frac{p_1-p}{p_1}\int_{\Omega_6}|Dh|^{p} dx,
\end{align*}
and then by \eqref{48}, we have
\begin{align}\label{49}
\int_{\Omega_5} |Dh-Df|^p dx  \leq \Bigg[\tau+C(\tau)\bigg(\underbrace{\int_{\Omega_5} (\lambda r|u-\overline{u}_{\Omega_5}|)^p dx}_{I_1}\bigg)^\frac{p_1-p}{p_1}\Bigg]\underbrace{\int_{\Omega_6}|Dh|^{p} dx}_{I_2}.
\end{align}

It follows from the triangle inequality that
\begin{align*}
I_1=\int_{\Omega_5} (\lambda r|u-\overline{u}_{\Omega_5}|)^p dx \leq&\ C \int_{\Omega_5} (\lambda r|u-h|)^p dx+C\int_{\Omega_5} (\lambda r|h-\overline{h}_{\Omega_5}|)^p dx\\
&+ C \int_{\Omega_5} (\lambda r|\overline{h}_{\Omega_5}-\overline{u}_{\Omega_5}|)^p dx\\
\leq&\ C \underbrace{\int_{\Omega_5} (\lambda r|u-h|)^p dx}_{I_3}+C\underbrace{\int_{\Omega_5} (\lambda r|h-\overline{h}_{\Omega_5}|)^p dx}_{I_4}.
\end{align*}
Remembering that $C_0|B_6|\left(\frac{M}{\varepsilon}\right)^p > (\lambda r)^p$ and using the Poincar\'e inequality and Lemma~\ref{Lem2}, we have
\begin{align*}
I_3=\int_{\Omega_5} (\lambda r|u-h|)^p dx &\leq C\int_{\Omega_5} (\lambda r|Du-Dh|)^p dx \\
&\leq C \left(\frac{1}{\varepsilon}\right)^p\int_{\Omega_5} |Du-Dh|^p dx \leq C \left(\frac{\varepsilon_1}{\varepsilon}\right)^p.
\end{align*}
On the other hand, Proposition~\ref{LemHOL} yields
\begin{align*}
I_4=\int_{\Omega_5} (\lambda r|h-\overline{h}_{\Omega_5}|)^p dx \leq C(\lambda r)^p\left(\frac{5}{\sigma}\right)^{p\beta}\|h\|^p_{L^\infty(B_\sigma)} \leq \frac{C}{\sigma^{p\beta}},
\end{align*}
because of $\|h\|_{L^\infty(B_\sigma)}\leq \frac{M}{\lambda r}.$ Consequently,
\begin{align*}
I_1 \leq C\left(\left(\frac{\varepsilon_1}{\varepsilon}\right)^p+\frac{1}{\sigma^{p\beta}}\right).
\end{align*}

Further on, using Lemma~\ref{Lem2} and \eqref{31}, we find
$$
I_2=\int_{\Omega_6}|Dh|^{p} dx \leq C\left(\int_{\Omega_6}|Du-Dh|^{p} dx + \int_{\Omega_6}|Du|^{p} dx\right) \leq C(\varepsilon_1^p+1) \leq C_1,
$$
while \eqref{49} gives
\begin{align*}
\int_{\Omega_5} |Dh-Df|^p dx &\leq C_1\left(\tau+C_2(\tau)\left(\left(\frac{\varepsilon_1}{\varepsilon}\right)^p+\frac{1}{\sigma^{p\beta}}\right)^\frac{p_1-p}{p_1}\right)\\
&\leq C_1\tau+C_1C_2(\tau)\left(\frac{\varepsilon_1}{\varepsilon}\right)^\frac{p(p_1-p)}{p_1}+C_1C_2(\tau)\left(\frac{1}{\sigma}\right)^\frac{p\beta(p_1-p)}{p_1}.
\end{align*}
Taking $\tau,$ $\varepsilon_1$ sufficiently small and $\sigma$ sufficiently large such that
$$C_1 \tau = \frac{\varepsilon^p}{3},\quad C_1C_2(\tau)\left(\frac{\varepsilon_1}{\varepsilon}\right)^\frac{p(p_1-p)}{p_1}\leq \frac{\varepsilon^p}{3} \quad \mathrm{and}\quad C_1C_2(\tau)\left(\frac{1}{\sigma}\right)^\frac{p\beta(p_1-p)}{p_1} \leq \frac{\varepsilon^p}{3},$$
we get the claim.

\medskip

$\mathbf{Case\ 2:}\ p\geq2.$
By using \eqref{Mo} instead of Lemma~\ref{Lem1} in the above proof,  we can obtain the conclusion in a similar manner.
\end{proof}

\begin{lemma}\label{Lem4}
Under the hypotheses of Lemma $\ref{Lem3},$ we further assume \eqref{33} and
 \eqref{34}. Then there exists a weak solution $v\in W^{1,p}(U)$ of \eqref{VP} such that
\begin{align*}
\|D\bar{v}\|_{L^\infty(\Omega_3)}\leq N_0\quad \mathrm{and}\quad\int_{\Omega_4} |Df-D\bar{v}|^p dx \leq \varepsilon^p
\end{align*}
for some constant $N_0=N_0(\gamma,n,p)>1.$ Here, the function $\bar{v}\in W^{1,p}(\Omega_4)$ is equal to
$v$ if $U=B_4,$ and $\bar{v}$ is the zero extension of $v$ from $B_4^+$ to $B_4$ if $U=B_4^+.$
\end{lemma}

\begin{proof}
According to Lemma~\ref{Lem2}, Lemma~\ref{Lem3} and \eqref{31}, it follows from the triangle inequality that
$$
\int_{\Omega_5} |Df|^p \ dx \leq C \int_{\Omega_5} |Du|^p + |Du-Dh|^p + |Dh-Df|^p  dx \leq C,
$$
where $C$ depends only on $n$ and $p.$ Then we proceed in doing a comparison estimate from standard perturbation argument, as in Lemma~3.1 and Lemma~3.7 of \cite{BR}, in order to obtain the desired conclusion.
\end{proof}

\begin{proof}[Proof of Lemma \ref{MainLem}]
The proof follows directly from the triangle inequality and Lemmas~\ref{Lem2}, \ref{Lem3} and \ref{Lem4}.
\end{proof}

\section{Global gradient estimates}\label{Sec5}

This section is devoted to the proof of the main result, Theorem~\ref{Thm1}. We start with a modified Vitali covering lemma for the problem \eqref{1}.

\begin{proposition}
{\em (see \cite{BW1, NP})}
\label{prop1}
Let $\mathcal{C}$ and $\mathcal{D}$ be measurable sets with $\mathcal{C}\subset \mathcal{D}\subset \Omega.$ Assume that $\Omega$ is $(\delta,R)$-Reifenberg flat.  Suppose that there exist $0< \varepsilon<1$ and $\sigma >1$ for which
\begin{enumerate}
\item $|\mathcal{C}| <\varepsilon |B_{R/\sigma}|;$
\item for all $x\in \Omega$ and $r\in (0,\frac{R}{\sigma}]$ with $|\mathcal{C}\cap B_r(x)|\geq \varepsilon|B_r(x)|,$ there holds $\Omega\cap B_r(x)\subset \mathcal{D}.$
\end{enumerate}
Then we have
$$|\mathcal{C}|\leq \left(\frac{10}{1-\delta}\right)^n\varepsilon|\mathcal{D}|.$$
\end{proposition}

We now return to the scaled and normalized problem \eqref{UP}.

\begin{lemma}\label{Lem10}
Assume that $\widetilde{\ba}$ satisfies \eqref{S} and \eqref{S1}. Let $\widetilde{u}\in W^{1,p}(\widetilde{\Omega})$ be a bounded weak solution of \eqref{UP} with $\|\widetilde{u}\|_{L^\infty(\widetilde{\Omega})}\leq \frac{M}{\lambda r}.$ Then there exists a constant $N_1=N_1(\gamma, n, p)>1$ so that for any small $\varepsilon\in(0,1),$ there exist a small constant $\delta=\delta(\gamma, \alpha, n, p, \Gamma, M, \varepsilon)>0$ and a large constant $\sigma=\sigma(\gamma, \alpha, n, p, \Gamma, M, \varepsilon)> 6$ such that if $\widetilde{\ba}(x,z,\xi)$ is $(\delta, \sigma)$-vanishing and $\widetilde{\Omega}$ is $(\delta, \sigma)$-Reifenberg flat, and if
\begin{equation}
\label{51}
\left\{x\in\widetilde{\Omega}_1:\m(|D\widetilde{u}|^p) \leq \left(\frac{6}{7}\right)^n\right\}
\cap \left\{x\in\widetilde{\Omega}_1:\m(|\widetilde{F}|^p)\leq\left(\frac{6}{7}\right)^n \delta^p\right\}\neq\varnothing
\end{equation}
then
$$
\left|\left\{x\in\widetilde{\Omega}_1: \m(|D\widetilde{u}|^p)(x)>\left(\frac{6}{7}\right)^n N_1^p\right\}\right|<\varepsilon|B_1|.
$$
\end{lemma}

\begin{proof}
Remembering \eqref{51}, there is a point $\widetilde{x}\in \widetilde{\Omega}_1$ such that for all $\rho>0,$
\begin{equation}\label{52}
\frac{1}{|B_\rho|}\int_{\widetilde{\Omega}_\rho(\widetilde{x})} |D\widetilde{u}|^p dx \leq \left(\frac{6}{7}\right)^n\quad\mathrm{and} \quad \frac{1}{|B_\rho|}\int_{\widetilde{\Omega}_\rho(\widetilde{x})} |\widetilde{F}|^p dx \leq \left(\frac{6}{7}\right)^n\delta^p.
\end{equation}
Let $\sigma>6.$ Since $\widetilde{\Omega}_\sigma \subset\widetilde{\Omega}_{\sigma+1}(\widetilde{x}),$
we have
$$
\frac{1}{|B_\sigma|}\int_{\widetilde{\Omega}_\sigma}|\widetilde{F}|^p\; dx \leq \left(\frac{7}{6}\right)^n\frac{1}{|B_{\sigma+1}|}\int_{\widetilde{\Omega}_{\sigma+1}(\widetilde{x})}|\widetilde{F}|^p\; dx \leq \delta^p.
$$
Similarly, it follows that
$$
\frac{1}{|B_\sigma|}\int_{\widetilde{\Omega}_\sigma}|D\widetilde{u}|^p\; dx \leq 1 \quad \mathrm{and}\quad \frac{1}{|B_6|}\int_{\widetilde{\Omega}_6}|D\widetilde{u}|^p\; dx \leq 1.
$$
Thus, we are under the hypotheses of 
Lemma~\ref{MainLem}, which implies that there exist a big constant $\sigma=\sigma(\gamma,\alpha, n,p,\Gamma,M,\varepsilon)> 6$ and a small constant $\delta=\delta(\gamma,\alpha, n,p,\Gamma,M,\varepsilon)>0$ such that the conclusion of Lemma \ref{MainLem}
holds for such $\bar{v}$ and $N_0.$

Further on, we will show that there exists a constant $N_1=N_1(\gamma,n,p)>1$ such that
\begin{equation}\label{53}
\left\{ y\in \widetilde{\Omega}_1:\m(|D\widetilde{u}|^p)>\left(\frac{6}{7}\right)^n N_1^p\right\}
\subset \left\{y\in \widetilde{\Omega}_1:\m_{\widetilde{\Omega}_4}(|D\widetilde{u}-D\bar{v}|^p)>N_0^p\right\}.
\end{equation}
To do this, let $\widetilde{y}\in\{y\in \widetilde{\Omega}_1:\m_{\widetilde{\Omega}_4}(|D\widetilde{u}-D\bar{v}|^p)\leq N_0^p\}.$ Then
$$\frac{1}{|B_\rho|}\int_{B_\rho(\widetilde{y})} \chi_{\widetilde{\Omega}_4}|D\widetilde{u}-D\bar{v}|^p\; dx \leq N_0^p $$
for any $\rho>0.$ If $\rho>2,$ then $\widetilde{\Omega}_\rho(\widetilde{y})\subset \widetilde{\Omega}_{2\rho}(\widetilde{x})$ and it follows from \eqref{52} that
\begin{align*}
\frac{1}{|B_\rho|}\int_{\widetilde{\Omega}_\rho(\widetilde{y})} |D\widetilde{u}|^p\; dx \leq \frac{1}{|B_\rho|} \int_{\widetilde{\Omega}_{2\rho}(\widetilde{x})} |D\widetilde{u}|^p\; dx \leq 2^n.
\end{align*}
On the other hand, if $\rho\in(0,2],$ then $\widetilde{\Omega}_\rho(\widetilde{y})\subset \widetilde{\Omega}_3$ and so we have
\begin{align*}
\frac{1}{|B_\rho|}\int_{\widetilde{\Omega}_\rho(\widetilde{y})} |D\widetilde{u}|^p\; dx &\leq 2^{p-1}\frac{1}{|B_\rho|}\int_{\widetilde{\Omega}_\rho(\widetilde{y})} |D\widetilde{u}-D\bar{v}|^p+|D\bar{v}|^p\; dx \\
 &\leq 2^{p-1}N_0^p + 2^{p-1}\frac{1}{|B_\rho|}\int_{\widetilde{\Omega}_\rho(\widetilde{y})} |D\bar{v}|^p\; dx \leq (2N_0)^p.
\end{align*}
Taking $N_1^p=\left(\frac{7}{6}\right)^n\max\left\{2^n, (2N_0)^p\right\},$ the claim \eqref{53} follows.

We now use \eqref{53}, the weak $(1,1)$-estimate for the Hardy--Littlewood maximal function and Lemma \ref{MainLem}, to observe that
\begin{align*}
\left|\left\{y\in\widetilde{\Omega}_1: \m(|D\widetilde{u}|^p)>\left(\frac{6}{7}\right)^n N_1^p\right\}\right| &\leq \left|\left\{y\in \widetilde{\Omega}_1:\m_{\widetilde{\Omega}_4}(|D\widetilde{u}-D\bar{v}|^p)>N_0^p\right\}\right|\\
&\leq  \frac{C(n,p)} {N_0^p} \int_{\widetilde{\Omega}_4} |D\widetilde{u}-D\bar{v}|^2\; dx \leq C \varepsilon^p|B_1|.
\end{align*}
Thus, the claim follows in view of the arbitrariness of $\varepsilon>0.$
\end{proof}

Turning back to the problem \eqref{1}, scaling and normalization give
\begin{corollary}\label{Cor1}
Assume that $\ba$ satisfies \eqref{3} and \eqref{4}. Let $u\in W^{1,p}(\Omega)$ be a bounded weak solution of \eqref{1} with $\|u\|_{L^\infty(\Omega)}\leq M.$ Then for any small constant $\varepsilon\in(0,1),$ there exist a small constant $\delta=\delta(\gamma,\alpha, n,p,\Gamma,M,\varepsilon)>0$ and a big constant $\sigma=\sigma(\gamma, \alpha, n, p, \Gamma, M, \varepsilon)> 6$ such that if $\ba(x,z,\xi)$ is $(\delta, R)$-vanishing and $\Omega$ is $(\delta, R)$-Reifenberg flat, and if
$$
\left\{x\in\Omega:\m(|Du|^p) \leq \left(\frac{6}{7}\right)^n\lambda^p\right\} \cap\left\{x\in\Omega:\m(|F|^p)\leq \left(\frac{6}{7}\right)^n \lambda^p\delta^p\right\}\neq\varnothing
$$
then
$$
\left|\left\{x\in\Omega_r(y): \m(|Du|^p)(x)>\left(\frac{6}{7}\right)^n\lambda^p N_1^p\right\}\right|<\varepsilon|B_r|
$$
for any $r\in\left(0,\frac{R}{\sigma}\right]$ and any $y\in \Omega.$
\end{corollary}

We now take $N_1,$ $\varepsilon$ and the corresponding $\delta$ and $\sigma$ from Corollary \ref{Cor1}.

\begin{lemma}\label{Lem11}
Assume that $\Omega$ is $(\delta,R)$-Reifenberg flat, and $\ba$ satisfies \eqref{3}, \eqref{4} and \eqref{5}. Let $F \in L^p(\Omega,\mr^n)$ and $u\in W^{1,p}_0(\Omega)$ be a bounded weak solution of \eqref{1} with $\|u\|_{L^\infty(\Omega)}\leq M.$ Then
\begin{align*}
&\left|\left\{ x\in\Omega: \m(|Du|^p) (x)> \left(\frac{6}{7}\right)^n N_1^{p(k+1)} \right\}\right|\\
& \qquad  \leq 20^n\varepsilon \left(\left|\left\{x\in\Omega:\m(|Du|^p)(x) > \left(\frac{6}{7}\right)^n N_1^{pk}\right\}\right|\right.\\
 &   \qquad\qquad\qquad +\left.\left|\left\{x\in\Omega:\m(|F|^p)(x)> \left(\frac{6}{7}\right)^n N_1^{pk}\delta^p\right\}\right|\right)
\end{align*}
for all integer $k\geq k_0,$  where $k_0$ is an integer satisfying
\begin{equation}
\label{55}
\frac{C_4}{N_1^{p(k_0+1)}}\int_\Omega |F|^p\;dx < \varepsilon|B_{R/\sigma}| \leq \frac{C_4}{N_1^{pk_0}}\int_\Omega |F|^p\;dx
\end{equation}
with $C_4=C_4(n,p,\gamma).$
\end{lemma}

\begin{proof}
Taking $u$ as a test function for \eqref{1}, we have
\begin{align*}
\int_{\Omega} \langle \ba(x,u,Du)-\ba(x,u,0),Du\rangle\; dx &= \int_{\Omega} \langle \ba(x,u,Du),Du\rangle\; dx\\
&= \int_{\Omega} \langle |F|^{p-2}F,Du\rangle\; dx\\
&\leq \int_{\Omega} |F|^{p-1}|Du|\; dx\\
&\leq \tau\int_{\Omega} |Du|^p\; dx + C(\tau)\int_{\Omega} |F|^p\; dx.
\end{align*}
Lemma~\ref{Lem1} and \eqref{Mo} give
\begin{align*}
\int_\Omega |Du|^p\; dx &\leq C_3\int_{\Omega} \langle \ba(x,u,Du)-\ba(x,u,0),Du\rangle\; dx\\
&\leq C_3\tau\int_{\Omega} |Du|^p\; dx + C(\tau)\int_{\Omega} |F|^p\; dx
\end{align*}
and selecting $\tau=\frac{1}{2C_3},$ we obtain
\begin{align*}
\int_\Omega |Du|^p\; dx  \leq C \int_{\Omega} |F|^p\; dx.
\end{align*}
This estimate and the weak type $(1,1)$-estimate for the maximal function yield
\begin{align*}
\left|\left\{x\in\Omega:\m(|Du|^p)(x)>\left(\frac{6}{7}\right)^n N_1^{p(k+1)}\right\}\right|  \leq &\ \frac{C}{N_1^{p(k+1)}}\int_\Omega |Du|^p\; dx \\
 \leq &\ \frac{C_4}{N_1^{p(k+1)}}\int_\Omega |F|^p\; dx,
\end{align*}
for some positive constant $C_4=C_4(n,p,\gamma).$ Selecting the integer $k_0$ for which \eqref{55} holds, we find that for all $k\geq k_0,$
\begin{align*}
\left|\left\{x\in\Omega:\m(|Du|^p)(x)>\left(\frac{6}{7}\right)^n N_1^{p(k+1)}\right\}\right|  \leq &\ \frac{C_4}{N_1^{p(k+1)}}\int_\Omega |F|^p\; dx \\
 \leq &\ \frac{C_4}{N_1^{p(k_0+1)}}\int_\Omega |F|^p\; dx < \varepsilon|B_{R/\sigma}| .
\end{align*}
We define now
$$
\mathcal{C}=\left\{ x\in\Omega: \m(|Du|^p)(x)> \left(\frac{6}{7}\right)^n N_1^{p(k+1)} \right\}
$$
and
\begin{align*}
\mathcal{D}=&\ \left\{ x\in\Omega:\m(|Du|^p)(x) > \left(\frac{6}{7}\right)^n N_1^{pk}\right \}\\
&\ \cup\left\{x\in\Omega:\m(|F|^p)(x)> \left(\frac{6}{7}\right)^n N_1^{pk}\delta^p\right\}
\end{align*}
in order to apply Proposition~\ref{prop1}. Then the first assumption of Proposition~\ref{prop1} follows directly, while the second one comes from Corollary~\ref{Cor1}. Consequently, we obtain the conclusion of Lemma~\ref{Lem11}.
\end{proof}

With all these tools in hand, we are in a position now to prove  Theorem~\ref{Thm1}.

\begin{proof}[Proof of Theorem~\ref{Thm1}]
Straightforward calculations yield
\begin{align*}
\int_\Omega \m(|Du|^p)^q\ dx =&\ q\int_0^\infty t^{q-1}|\{x\in\Omega:\m(|Du|^p)(x)>t\}|\ dt\\
\leq&\ q\int_0^{\left(\frac{6}{7}\right)^n N_1^{pk_0}} t^{q-1}|\{x\in\Omega:\m(|Du|^p)(x)>t\}|\ dt\\
&\  +q\sum_{k=k_0}^\infty\int_{\left(\frac{6}{7}\right)^n N_1^{pk}}^{\left(\frac{6}{7}\right)^n N_1^{p(k+1)}} t^{q-1}|\{x\in\Omega:\m(|Du|^p)(x)>t\}|\ dt \\
\leq&\ \left(\frac{6}{7}\right)^{nq} N_1^{pqk_0}|\Omega|\\
&\  +\left(\frac{6}{7}\right)^{nq}(N_1^{pq}-1)\sum_{k=k_0}^\infty N_1^{pqk}\left|\left\{x\in\Omega:\m(|Du|^p)(x)>\left(\frac{6}{7}\right)^n N_1^{pk}\right\}\right|,
\end{align*}
where $N_1$ is given in Lemma \ref{Lem10} and $k_0$ is given in Lemma \ref{Lem11}.

Keeping in mind \eqref{55}, we have
\begin{equation}
\label{M1}
N_1^{pqk_0} \leq \left(\frac{C}{\varepsilon|B_{R/\sigma}|}\int_\Omega |F|^p\;dx\right)^q \leq C \int_\Omega |F|^{pq}\;dx.
\end{equation}
Further on, Lemma~\ref{Lem11} yields
\begin{align*}
\sum_{k=k_0}^\infty N_1^{pqk}\bigg|\bigg\{&x\in\Omega:\m(|Du|^p)(x)>\left(\frac{6}{7}\right)^n N_1^{pk}\bigg\}\bigg|\\
\leq &\ \sum_{k=k_0}^\infty N_1^{pqk}(20^n\varepsilon)^{k-k_0}\bigg|\bigg\{x\in\Omega: \m(|Du|^p)(x)>\left(\frac{6}{7}\right)^n N_1^{pk_0}\bigg\}\bigg|\\
&\ +\sum_{k=k_0}^\infty\sum_{l=k_0}^k N_1^{pqk}(20^n\varepsilon)^{k-l}\bigg|\bigg\{x\in\Omega:\m(|F|^p)(x)>\left(\frac{6}{7}\right)^n N_1^{pl}\delta^p\bigg\}\bigg|\\
= &\ \sum_{k=k_0}^\infty N_1^{pqk}(20^n\varepsilon)^{k-k_0}\bigg|\bigg\{x\in\Omega: \m(|Du|^p)(x)>\left(\frac{6}{7}\right)^n N_1^{pk_0}\bigg\}\bigg|\\
&\ +\sum_{l=k_0}^\infty\sum_{k=l}^{\infty} N_1^{pqk}(20^n\varepsilon)^{k-l}\bigg|\bigg\{x\in\Omega:\m(|F|^p)(x)>\left(\frac{6}{7}\right)^n N_1^{pl}\delta^p\bigg\}\bigg|\\
=&: S_1+S_2.
\end{align*}

We take now $\varepsilon>0$  small enough to have $0 < 20^nN_1^{pq}\varepsilon<\frac{1}{2},$ and observe that \eqref{M1} gives
\begin{align*}
S_1=&\ \sum_{k=k_0}^\infty N_1^{pqk}(20^n\varepsilon)^{k-k_0}\bigg|\bigg\{x\in\Omega: \m(|Du|^p)(x)>\left(\frac{6}{7}\right)^n N_1^{pk_0}\bigg\}\bigg|\\
\leq&\ 2 N_1^{pqk_0}|\Omega| \leq 2|\Omega|\left(\frac{C}{\varepsilon|B_{R/\sigma}|}\int_\Omega |F|^p\;dx\right)^q \leq C \int_\Omega |F|^{pq}\;dx.
\end{align*}
On the other hand,
\begin{align*}
S_2=&\ \sum_{l=k_0}^\infty\sum_{k=l}^{\infty} N_1^{pqk}(20^n\varepsilon)^{k-l}\bigg|\bigg\{x\in\Omega:\m(|F|^p)(x)> \left(\frac{6}{7}\right)^n N_1^{pl}\delta^p\bigg\}\bigg|\\
\leq&\  2 \sum_{l=k_0}^\infty N_1^{pql}\bigg|\bigg\{x\in\Omega:\m(|F|^p)(x)>\left(\frac{6}{7}\right)^n N_1^{pl}\delta^p\bigg\}\bigg|\\
\leq&\ C\int_\Omega \m(|F|^p)^q(x)\ dx.
\end{align*}
Therefore, we conclude that
$$
\int_\Omega \m(|Du|^p)^q(x)\ dx\leq C\int_\Omega |F|^{pq}(x)\ dx + C\int_\Omega \m(|F|^p)^q(x)\ dx.
$$
At this point, applying the strong type $(pq,pq)$-estimate for the maximal function, we complete the proof of Theorem~\ref{Thm1}.
\end{proof}

\section{Refinements of the gradient estimates}\label{Sec6}

It turns out that, for values of the exponent $p$ greater than or equal to the space dimension $n,$ the result of Theorem~\ref{Thm1} continues to hold under weaker assumption on the $z$-behaviour of $\ba(x,z,\xi)$ than the H\"older continuity \eqref{4}.

Precisely, assume hereafter that $p\geq n$ and
\begin{equation}\label{4'}
|\ba(x,z_1,\xi)-\ba(x,z_2,\xi)|\leq \omega(|z_1-z_2|)\xi|^{p-1}
\end{equation}
for a.a. $x\in\Omega,$ $\forall z_1,z_2\in\mr$ and $\forall\xi\in\mr^n,$ where $\omega\colon \mathbb{R}^+\to\mathbb{R}^+$ is a modulus of continuity, that is, $\lim_{\tau\to0^+}\omega(\tau)=0.$

We will make use of our results from \cite{BPS,BPS-arxiv} where global H\"older continuity is proved for the weak solutions of general quasilinear elliptic equations with Morrey data over domains with \textit{uniformly $p$-thick complements.} Since any Reifenberg flat domain has uniformly $p$-thick complement (cf. \cite{BPS-arxiv}), the above mentioned results hold true in the situation here considered.

Using the assumption \eqref{3} on uniform ellipticity and its outgrowth \eqref{Mo1}, it is not hard to check that
\begin{align}
\label{6.2}
\langle \ba(x,z,\xi)+|F(x)|^{p-2}F(x),\xi\rangle \geq&\ C_1|\xi|^p -C_2\big(|F(x)|^p \big)^{\frac{p}{p-1}},\\
\label{6.3}
\big| \ba(x,z,\xi)-|F(x)|^{p-2}F(x)\big| \leq &\ \gamma^{-1}|\xi|^{p-1} + |F(x)|^{p-1}
\end{align}
for a.a. $x\in\Omega,$ $\forall (z,\xi)\in \mr\times\mr^n,$ and where $C_1$ and $C_2$ are positive constants depending only on $n,$ $p$ and $\gamma.$

Further on, assuming $|F(x)|^{p}\in L^q(\Omega)$ with $q>1$ as did in Theorem~\ref{Thm1}, it follows $|F(x)|^{p-1}\in L^{\frac{pq}{p-1}}(\Omega).$ We have
$$
\frac{pq}{p-1}>\frac{p}{p-1}\geq \frac{n}{p-1}
$$
because of $q>1$ and $p\geq n,$ and therefore, we first choose a number
$$
p'\in \left(\frac{p}{p-1},\frac{pq}{p-1}\right)
$$
and consequently define
$$
\lambda'=n\left(1-p'\frac{p-1}{pq}\right).
$$
It is obvious that $\lambda'\in(0,n)$ and the choice of $p'$ and $\lambda'$ guarantees that the Lebesgue space $L^{\frac{pq}{p-1}}(\Omega)$ is embedded into the Morrey space $L^{p',\lambda'}(\Omega).$ Thus $|F(x)|^{p-1}\in L^{p',\lambda'}(\Omega)$ with
$$
p'>\frac{p}{p-1},\quad (p-1)p'+\lambda'>n,
$$
and the inequalities \eqref{6.2} and \eqref{6.3} ensure the validity of the results from \cite{BPS,BPS-arxiv}. In other words, the weak solution of the Dirichlet problem \eqref{1} is H\"older continuous function up to the boundary and
\begin{equation}\label{6.4}
\sup_{\overline{\Omega}}|u(x)|+ \sup_{x,y\in \overline\Omega,\ x\neq y} \frac{|u(x)-u(y)|}{|x-y|^\alpha} \leq H
\end{equation}
with $\alpha\in(0,1)$ and $H$ depending on the data of \eqref{1} and on $\|Du\|_{L^p(\Omega)}.$

Let us note at this step that the above arguments are relevant only if $p=n,$ because \eqref{6.4} is direct consequence of the Sobolev embeddings and the Morrey lemma when $p>n.$ Moreover, in case of \textit{Lipschitz continuous} domain $\Omega,$ \eqref{6.4} follows from the classical results of Ladyzhenskaya and Ural'tseva \cite[Chapter~IV]{LU}.

For a fixed weak solution $u\in W^{1,p}_0(\Omega)$ of \eqref{1}, we define now the Carath\'eodory map
$$
\mathbf{A}(x,\xi):= \ba (x,u(x),\xi)
$$
and use \eqref{5}, \eqref{4'} and \eqref{6.4} to infer, through \cite[Lemma~1]{PRS}, that it obeys the $(\delta,R)$-vanishing property
(see \cite[Definition~2.2]{BR}). Moreover,
\eqref{3} implies uniform ellipticity of $\mathbf{A}(x,\xi)$ and $u\in W^{1,p}_0(\Omega)$ solves
$$
\begin{cases}
\mathrm{div\,}\mathbf{A}(x,Du)=\mathrm{div\,}(|F|^{p-2}F) & \textrm{in}\ \Omega\\
u=0 & \textrm{on}\ \partial\Omega.
\end{cases}
$$

This way, \cite[Theorem~2.6]{BR} yields the following refinement of Theorem~\ref{Thm1} in the case $p\geq n$ where the H\"older continuity \eqref{4} of $\ba(x,z,\xi)$ with respect to $z$ is relaxed to only continuity. The constant $H$ below is the one appearing in \eqref{6.4}.

\begin{theorem}
Suppose \eqref{3} and \eqref{4'}, and let $u\in W^{1,p}_0(\Omega)$ be a weak solution of \eqref{1} with $p\geq n.$ Let $|F|^p\in L^q(\Omega)$ for some $q\in (1,\infty).$ Then there is a small constant $\delta=\delta(\gamma,n,p,q,\omega(\cdot),H)$ such that if $\ba$ is $(\delta,R)$-vanishing and $\Omega$ is $(\delta,R)$-Reifenberg flat, then $|Du|^p\in L^q(\Omega)$ with the estimate
$$
\int_\Omega |Du|^{pq} dx \leq C \int_\Omega |F|^{pq} dx,
$$
where $C=C(n,p,q,\gamma,\omega(\cdot),H,|\Omega|).$
\end{theorem}

\end{document}